\documentclass[11pt]{amsart}
\usepackage{amsmath,amssymb,amsthm}
\usepackage{graphicx}
\usepackage{color}
\usepackage{geometry}
\usepackage{setspace}

\allowdisplaybreaks

\numberwithin{equation}{section}
\nonstopmode

\theoremstyle{plain}
\newtheorem{corollary}[equation]{Corollary}

\newtheorem{theorem}[equation]{Theorem}
\newtheorem{conjecture}[equation]{Conjecture}
\newtheorem{lemma}[equation]{Lemma}
\newtheorem{proposition}[equation]{Proposition}
\newtheorem{problem}[equation]{Problem}

\theoremstyle{definition}

\newtheorem{remark}[equation]{Remark}

\newtheorem{example}[equation]{Example}

\newtheorem{mysubsection}[equation]{}

\begin{document}


\title{The Ptolemy-Alhazen problem and quadric 
           surface mirror reflection}

\title[Alhazen's problem and conicoid surface mirror reflection]
        {\bf The Ptolemy-Alhazen problem and quadric 
           surface mirror reflection}
\author[M. Fujimura]{Masayo Fujimura}
\address{Department of Mathematics, National Defense Academy of Japan, Japan}
\email{masayo@nda.ac.jp}

\author[M. Mocanu]{Marcelina Mocanu}
\address{Department of Mathematics and Informatics, Vasile Alecsandri
         University of Bac\u{a}u, Romania}
\email{mmocanu@ub.ro}

\author[M. Vuorinen]{Matti Vuorinen}
\address{Department of Mathematics and Statistics, University of Turku,
         Turku, Finland}
\email{vuorinen@utu.fi}


\date{March 22, 2022: fmv20220322c.tex}

\begin{abstract}
  We discuss the problem of the reflection of light on 
  spherical and quadric surface mirrors. 
  In the case of spherical mirrors, this problem is known as 
  {the Alhazen} problem.
 {
  For the spherical mirror problem,
  we focus on the reflection property of an ellipse,
  and show that the catacaustic curve of the unit circle follows naturally
  from the equation obtained from the reflection property of an ellipse.
  }
  Moreover, we provide an algebraic equation that solves
  Alhazen's problem for quadric surface mirrors.
\end{abstract}

\maketitle

\section{Introduction}
Alhazen's problem {\cite[p.1010]{gbl}}
is the problem that asks the following: 
{\it Given a light source and a spherical mirror, 
find the point on the mirror where the light will be reflected 
to the eye of an observer}.
{This problem is first formulated by Ptolemy in 150 AD
and is therefore also called the Ptolemy-Alhazen problem.}
We call {the} reflection point
{of this problem} the {\it PA-point}.

This problem is equivalent to solving the following problem for a disk.
For given two points 
$ z_1,\ z_2 \in\mathbb{D}=\{z\in\mathbb{C}\,:\, \vert z\vert <1\}$,
find $ u \in\partial\mathbb{D} $ such that
\[
    \vert \angle(z_1,u,0)\vert =\vert \angle(0,u,z_2)\vert .
\]

Many mathematicians and researchers of geometrical optics 
have investigated this problem. For a short
history of this topic, see \cite{gbl}. 
Some of the recent studies of this topic,
from the point of astrophysics and signal transmission 
are \cite{mavreas, mibama}.
The bibliographies of \cite{fhmv,fmv} include several pointers to
the literature.

As above, this problem has a long history, but it has finally been 
solved algebraically only recently.
Elkin \cite{elkin} found in 1965 an equation of degree 4 
solving this problem.

His strategy was to find the PA-point as the intersection of 
the unit circle and a circle centered at $ z_1 $.
In this paper, we will study algebraic equations that solve 
reflection problems on spherical and quadric surfaces.

In \cite{fhmv}, we studied Alhazen's problem and its relation to 
the triangular ratio metric $ s_G $ of a given domain $ G \subset \mathbb{C} $
defined as follows 
(see, for instance {\cite{hkv-book}},
\cite{rainio}, \cite{chkv}, \cite{hvz} and {\cite{m}} 
for other studies on the triangular ratio metric)
\begin{equation}\label{eq:sG}
  s_G(z_1,z_2)=\sup_{z\in\partial G}
        \frac{\vert z_1-z_2\vert }{\vert z_1-z\vert +\vert z-z_2\vert },
         \quad  z_1,z_2\in G .
\end{equation}
In \cite{fhmv}, we discussed an equation that uses the reflection property
of an ellipse. 

Solutions to Alhazen's problems for quadric surfaces have also
found applications to many other fields besides mathematics.
Agrawal, Taguchi, and Ramalingam 
{\cite{atr,atr2}}  
studied Alhazen's problem for application to a camera 
with quadric shaped mirrors.
They constructed an equation with six roots
that include the PA-points in \cite{atr2}.
In \cite{mibama}, Miller, Barnes, and MacKenzie studied
an equation solving Alhazen's problem and
proposed a fast method for choosing the correct point from the roots.
In addition, they mentioned that if their method could be extended 
to the case of elliptical surfaces, it could be useful for GPS 
communication, and could also be applied to computer rendering.
This motivated us to construct an algebraic equation that
{yields} the PA-point
for quadric surfaces.

\smallskip

This paper is organized as follows.
In Section \ref{sect:disc}, we discuss the relation between the equation
using the reflection property of an ellipse and 
the catacaustic curve of the circle.
We can also use the properties of the circle of Apollonius 
to construct an equation that solves Alhazen's problem. 
This equation is studied in Section \ref{sect:apollo}.
In Section \ref{sect:conic}, we discuss Alhazen's problem for 
quadric surfaces and provide 
an equation $ F_4=0 $ that gives the PA-points in Theorem \ref{lem:tan}.
This equation is different from the one formulated by Agrawal et al. 
in \cite[Section 2]{atr2}.
In fact, the algebraic equation $ F_4=0 $ is obtained by using 
the reflective property of an ellipse and a method based on 
algebraic geometry.
Moreover, using Theorem \ref{lem:tan}, we give the calculation method of the
triangular ratio metric on conic domains in Theorem \ref{thm:6ji}.
The application to the calculation of the triangular ratio metric is 
also discussed in Section \ref{sect:smetric}.

In this paper, several symbolic computation systems are used effectively.
For graphics, we use GeoGebra\footnote{{\tt https://www.geogebra.org/}},
dynamic mathematics software,
to create Figures \ref{fig:apollo}, \ref{pic:1} and \ref{pic:3},
whereas
Figures \ref{fig:caustic} and \ref{fig:contour1} are drawn 
using Mathematica\footnote{{\tt https://www.wolfram.com/}}
and Risa/Asir\footnote{{\tt http://www.math.kobe-u.ac.jp/Asir/asir.html}
(Kobe Distribution)}, symbolic computation systems,
respectively.
We also use the techniques of computer algebra 
such as resultant to obtain some target equations.
In particular, we use Risa/Asir to obtain the result 
of Theorem \ref{lem:tan}, 
which is difficult to calculate manually.


\section{Alhazen's problem on a disk --Solution using ellipses}
\label{sect:disc}
In this section,  we consider Alhazen's problem on the unit disk.

For $ z_1,\ z_2\in\mathbb{D} $, 
let $ z $ be the PA-point on $ \partial\mathbb{D} $ with respect to
these two points.

\begin{mysubsection}{\bf Solution using ellipses}

Using the reflective property of an ellipse,
the PA-point can be found as the points of tangency of 
an ellipse with foci $ z_1,\, z_2 $ and the unit circle
(see the left figure in Figure \ref{fig:apollo}).
\end{mysubsection}

\begin{lemma}[{See, for example \cite[Theorem 1.1]{fhmv}}]
\label{lem:pa-eq}
  For $ z_1,z_2\in\mathbb{D} $, the PA-point $ z $ is given as a solution of 
  the equation
  \begin{equation}\label{eq:reflect}
     \overline{z_1}\overline{z_2} z^4-(\overline{z_1}+\overline{z_2})z^3+(z_1+z_2)z-z_1z_2=0.
  \end{equation}
\end{lemma}

This lemma is also valid for the external reflection, i.e.
the equation \eqref{eq:reflect} holds for 
$ z_1,\, z_2\in\mathbb{C}\setminus\overline{\mathbb{D}} $
if the line segment $ [z_1,z_2] $ has no intersection with 
$ \partial{\mathbb{D}} $.
Note that  if $ [z_1,z_2]\cap\partial\mathbb{D}\neq \emptyset $,
the light is blocked by the boundary (mirror) and never reaches the observer.
Moreover, in this case
$ \sup_{z\in\partial\mathbb{D}}\{\vert z_1-z_2\vert 
    /(\vert z_1-z\vert +\vert z-z_2\vert )\}= 1 $.

\begin{mysubsection}{\bf Number of roots and catacaustic of a circle}

A root $u\in \partial \mathbb{D}$ of the equation 
\[
   P_{z_{1},z_{2}}(u)=\overline{z_{1}z_{2}}u^{4}
         -\left( \overline{z_{1}}+\overline{z_{2}}\right)u^{3}
         +\left( z_{1}+z_{2}\right) u-z_{1}z_{2}=0
\]
is a PA-point if and only if 
\begin{equation*}
  \mbox{Re}\left( \overline{z_{1}z_{2}}u^{2}-\left( \overline{z_{1}}
  +\overline{z_{2}}\right) u\right) +1>0,
\end{equation*}
as follows from \cite[Lemma 3.1]{fhmv}. If $z_{1},z_{2}\in \mathbb{%
D\setminus }\left\{ 0\right\} $, then the above inequality holds whenever $%
u\in \partial \mathbb{D}$, therefore all the roots of $P_{z_1,z_2}=0$
that lie on the unit circle, called \textit{unimodular} roots, are PA-points.

If $z_{1}=0$ and $z_{2}=\left\vert z_{2}\right\vert e^{i\alpha }\neq 0$,
then the roots of \eqref{eq:reflect} are $0$ and $\pm e^{i\frac{\alpha }{2}}$. In
the following we will assume $z_{1},z_{2}\in \mathbb{C\setminus }\left\{
0\right\} $, therefore \eqref{eq:reflect} is a quartic equation.

The equation $P_{z_{1},z_{2}}=0 $ has always at least two distinct unimodular
roots \cite[Lemma 2.4]{fhmv} and both cannot have multiplicity $2$ 
\cite[Lemma 4.1]{fhmv}. 
Moreover, $P_{z_{1},z_{2}}=0$ has four simple unimodular
roots if $z_{1},z_{2}\in \mathbb{C\setminus }\mathbb{D}$ . 
If $P_{z_1,z_2}=0$ has a
triple root $a$ and a simple root $b$, then $\left\vert a\right\vert =1$ and 
$b=-a$ \cite[Lemma 4.3]{fhmv}. We characterize in terms of $z_{1}$ and $%
z_{2} $ all the possible cases for the number of unimodular roots of 
$P_{z_{1},z_{2}}=0$ and their multiplicities, both algebraically and
geometrically. Various approaches to particular cases of
this problem are scattered
through the literature \cite{W}, \cite{DG}, \cite{fhmv}.

We will denote by $D(P)$ the discriminant of the complex 
polynomial $P$.
\end{mysubsection}
\begin{proposition}\label{rootnum}
  Let $z_{1},z_{2}\in \mathbb{C\setminus }\left\{ 0\right\} $
  and
  \[
    P_{z_{1},z_{2}}\left( u\right) =\overline{z_{1}z_{2}}u^{4}-\left( 
     \overline{z_{1}}+\overline{z_{2}}\right) u^{3}+\left( z_{1}+z_{2}\right)
     u-z_{1}z_{2}.
  \]
  Then
  \begin{enumerate}
    \item  \label{enu:1}
         $P_{z_{1},z_{2}}=0$ has four simple unimodular roots if and only if $%
         D\left( P_{z_{1},z_{2}}\right) <0$;
    \item  \label{enu:2}
         $P_{z_{1},z_{2}}=0$ has two simple unimodular roots and two 
         distinct roots
         off the unit circle if and only if $D\left( P_{z_{1},z_{2}}\right) >0$;
    \item \label{enu:3}
       $P_{z_{1},z_{2}}=0$ has at least one multiple root if and only if $D\left(
       P_{z_{1},z_{2}}\right) =0$.
  \end{enumerate}
  Moreover, in the case \eqref{enu:3} all the roots of 
     $P_{z_{1},z_{2}}=0$ are unimodular and $ P_{z_{1},z_{2}}=0$ 
  has either a double root and two simple roots, or a triple
  root $ v $ and $ -v $ as a simple root.
\end{proposition}

\begin{proof}
Denote $s=z_{1}+z_{2}=s_{1}+is_{2}$ and $p=z_{1}z_{2}=p_{1}+ip_{2}$, where $%
s_{1},s_{2}$ and $p_{1},p_{2}$ are real numbers. Then $P_{z_{1},z_{2}}(u)=%
\overline{p}u^{4}-\overline{s}u^{3}+su-p$.

As in \cite{W}, using the substitution $u=(1+it)/(1-it)$
 we see that \eqref{eq:reflect} turns into an algebraic equation with 
real coefficients. 
We have 
\begin{equation*}
  P_{z_{1},z_{2}}\left( \frac{1+it}{1-it}\right) 
    =\frac{(-2i)}{\left(1-it\right) ^{4}}Q_{z_{1},z_{2}}(t)\text{, }
      t\in \mathbb{C\setminus }\left\{-i\right\}
\end{equation*}%
where 
\begin{equation*}
  Q_{z_{1},z_{2}}\left( t\right) 
    =\left( s_{2}+p_{2}\right) t^{4}+2\left(s_{1}+2p_{1}\right) t^{3}
     -6p_{2}t^{2}+2\left( s_{1}-2p_{1}\right) t-\left(s_{2}-p_{2}\right) \text{. }
\end{equation*}

There is a one-to-one correspondence between the unimodular zeros of 
$P_{z_{1},z_{2}}=0$ different from $\left( -1\right) $ and the real roots of 
$Q_{z_{1},z_{2}}=0$, as $P_{z_{1},z_{2}}\left( e^{i\varphi }\right) =0$ if and
only if $Q_{z_{1},z_{2}}\left( \tan \frac{\varphi }{2}\right) =0$, where 
$\varphi \in \left( -\pi ,\pi \right) $.

Similarly, we consider the reciprocal polynomials $P_{z_{1},z_{2}}^{\ast }$
and $Q_{z_{1},z_{2}}^{\ast }$ of $P_{z_{1},z_{2}}$ and $Q_{z_{1},z_{2}}$,
respectively: 
\begin{align*}
  P_{z_{1},z_{2}}^{\ast }(u) 
      &= -pu^{4}+su^{3}-\overline{s}u+\overline{p}, \\
  Q_{z_{1},z_{2}}^{\ast }\left( t\right) 
      &= -\left( s_{2}-p_{2}\right)
         t^{4}+2\left( s_{1}-2p_{1}\right) t^{3}-6p_{2}t^{2}+2\left(
         s_{1}+2p_{1}\right) t+\left( s_{2}+p_{2}\right) .
\end{align*}
We have $P_{z_{1},z_{2}}^{\ast }\left( (t+i)/(t-i)\right) 
=(-2i)/( t-i)^{4}Q_{z_{1},z_{2}}^{\ast }(t)$, $t\in \mathbb{C\setminus 
}\left\{ i\right\} $. There is a one-to-one correspondence between the
unimodular zeros of $P_{z_{1},z_{2}}^{\ast }=0$ different from $1$ and the
real roots of $Q_{z_{1},z_{2}}^{\ast }=0$, 
as $P_{z_{1},z_{2}}\left(e^{i\varphi }\right) =0$ if and only if 
$Q_{z_{1},z_{2}}^{\ast }\left( \cot \frac{\varphi }{2}\right) =0$, 
where $\varphi \in \left( 0,2\pi \right) $.

Note that $Q_{z_{1},z_{2}}^{\ast }$ is a quartic polynomial if and only if $%
s_{2}-p_{2}\neq 0$, which is equivalent to $P_{z_{1},z_{2}}(1)\neq 0$.

We discuss the number of real roots of $Q_{z_{1},z_{2}}=0$ and 
$ Q_{z_{1},z_{2}}^{\ast }=0$. 
It is known that the discriminant of a quartic
equation with real coefficients is positive if and only if 
the equation has four simple roots that are either all real 
or two pairs of complex conjugates (see for example \cite{rees}).

\textbf{Case 1.} Assume that $s_{2}+p_{2}\neq 0$. Then $P_{z_{1},z_{2}}%
\left( -1\right) \neq 0$ and $Q_{z_{1},z_{2}}$ is a quartic polynomial. The
discriminant of the polynomial $P_{z_{1},z_{2}}$ is the real number 
\begin{equation*}
  D(P_{z_{1},z_{2}})=\frac{1}{27}\left( 4I^{3}-J^{2}\right) \text{, }
\end{equation*}%
where $I=-12\left\vert p\right\vert^{2}+3\left\vert s\right\vert^{2}$ 
and $ J=-27\left( s^{2}\overline{p}-\overline{s}^{2}p\right) $, 
hence 
\begin{align} \label{eq:DiscrP}
  D\left( P_{z_{1},z_{2}}\right)  
    &= 4\big( \left\vert s\right\vert^{2}
          -4\left\vert p\right\vert^{2}\big) ^{3}-27\left( s^{2}\overline{p}
          -\overline{s}^{2}p\right) ^{2} \\ \notag
    &= 4\left\vert s\right\vert^{6}+6\left\vert s\right\vert^{4}\left\vert
        p\right\vert^{2}+192\left\vert s\right\vert^{2}\left\vert  
        p\right\vert^{4}-256\left\vert p\right\vert^{6}
        -54\mbox{Re}\left( \overline{s}^{4}p^{2}\right) .
\end{align}
The above formula appears in \cite{W}. Using the properties of the
discriminant of a binary form as a projective invariant
\cite[Definition 2.2]{J} it follows that 
$D(Q_{z_{1},z_{2}})=-64D(P_{z_{1},z_{2}})$,
see also \cite{W}.

Since $P_{z_{1},z_{2}}=0$ has at least two distinct unimodular roots, 
$Q_{z_{1},z_{2}}=0$ has at least two distinct real roots. 
The discriminant $D\left( Q_{z_{1},z_{2}}\right) =0$ if and only 
$Q_{z_{1},z_{2}}=0$ has multiple roots. 
$D\left( Q_{z_{1},z_{2}}\right) >0$ if and only if all the
roots of $Q_{z_{1},z_{2}}=0$ are real, 
since the equation $Q_{z_1,z_2}=0$ can have at most 2 non-real roots.
Then $D\left( Q_{z_{1},z_{2}}\right) <0$ if and
only if $Q_{z_{1},z_{2}}=0$ has two distinct real roots and two conjugated
non-real roots.
For the relation between the roots of a quartic equation 
and the discriminant, see also \cite{rees}.
We obtain the following:
\begin{enumerate}
  \item[(i)] 
        $P_{z_{1},z_{2}}=0$ has four simple unimodular roots 
        if and only if all
        the roots of $Q_{z_{1},z_{2}}=0$ are real, which is equivalent to 
         $ D(Q_{z_{1},z_{2}})>0$, i.e. to $D(P_{z_1,z_2})<0$.
  \item[(ii)]
        $P_{z_{1},z_{2}}=0$ has two simple unimodular roots and 
        two distinct roots
        off the unit circle if and only if $Q_{z_{1},z_{2}}=0$ has 
        two distinct real roots and two non-real roots, which is 
        equivalent to $D(Q_{z_{1},z_{2}})<0$, i.e. to $D(P_{z_{1},z_{2}})>0$.
  \item[(iii)]
        $P_{z_{1},z_{2}}=0$ has at least one multiple root 
        if and only if $ D(P_{z_{1},z_{2}})=0$. 
        Assuming that $P_{z_{1},z_{2}}=0$ has at least one
        multiple root the claim is obtained by 
        \cite[Lemmas 4.2 and 4.3]{fhmv}.
\end{enumerate}

\textbf{Case 2.} Assume that $s_{2}-p_{2}\neq 0$. Then $P_{z_{1},z_{2}}%
\left( 1\right) =P_{z_{1},z_{2}}^{\ast }\left( 1\right) \neq 0$ and $%
Q_{z_{1},z_{2}}^{\ast }$ is a quartic polynomial.

We have $D\left( P_{z_{1},z_{2}}\right) 
=D\left( P_{z_{1},z_{2}}^{\ast}\right) $ and the discriminants of 
$P_{z_{1},z_{2}}^{\ast }$ and $Q_{z_{1},z_{2}}^{\ast }$ are related by 
$D(Q_{z_{1},z_{2}}^{\ast})=-64D(P_{z_{1},z_{2}}^{\ast })$. 

The discussion continues as in Case 1, replacing $P_{z_{1},z_{2}}$ by 
$P_{z_{1},z_{2}}^{\ast }$ and $Q_{z_{1},z_{2}}$ by $Q_{z_{1},z_{2}}^{\ast }$.

\textbf{Case 3. } 
The remaining case $ s_2=p_2=0 $.
We can assume that $ s $ and $ p $ are real numbers.
In this case $P_{z_1,z_2}(u)=\left(u-1\right) \left( u+1\right) 
\left( pu^{2}-su+p\right) $.
The polynomial $P_{z_1,z_2}$
has only real coefficients and $D(P_{z_1,z_2})=4\left( s^{2}-4p^{2}\right) ^{3}$.
The roots of $P_{z_1,z_2}=0$ are $1$, 
$\left( -1\right) $ and the roots of $R\left( u\right)=pu^{2}-su+p$.

If $D(P_{z_1,z_2})>0$, i.e. $\left\vert s\right\vert >2\left\vert p\right\vert $,
then the roots of $R=0$ are $(s\pm \sqrt{s^{2}-4p^{2}})/(2p)$ and cannot
be unimodular.

If $\left\vert s\right\vert <2\left\vert p\right\vert $, then the roots 
$(s\pm i\sqrt{4p^{2}-s^{2}})/(2p)$ of $R=0$ are unimodular.

If $\left\vert s\right\vert =2\left\vert p\right\vert $, then the double
root $s/(2p)$ of $R=0$ belongs to $\left\{ \pm 1\right\} $.

The claims \eqref{enu:1}, \eqref{enu:2} and \eqref{enu:3}
follow straightforward in this case. Moreover,
if $D(P_{z_1,z_2})=0$, then $P_{z_1,z_2}$ has a triple root and one simple root.
\end{proof}

Using Proposition \ref{rootnum} we get a purely algebraic proof of \cite[%
Proposition 4.5]{fhmv}, whose original proof used Complex Analysis arguments
and results on self-inversive polynomials. 

\begin{corollary}\label{diskroots}
  Let $z_{1},z_{2}\in \mathbb{C\setminus }\left\{ 0\right\} $. 
  Denote 
  \begin{gather*}
   P_{z_{1},z_{2}}(u)=\overline{z_{1}z_{2}}u^{4}-\left( \overline{z_{1}}
     +\overline{z_{2}}\right) u^{3}+\left( z_{1}+z_{2}\right) u-z_{1}z_{2}, \\
     E_{1}\left( z_{1},z_{2}\right) =\left\vert z_{1}+z_{2}\right\vert
        -\left\vert z_{1}z_{2}\right\vert 
            \mbox{\quad and \quad}
   E_{2}\left( z_{1},z_{2}\right)=\left\vert z_{1}+z_{2}\right\vert 
   -2\left\vert z_{1}z_{2}\right\vert .
  \end{gather*}
  Assume that
    $\mbox{{\rm Im}}  
      \big(\left( z_{1}+1\right) \left( z_{2}+1\right)\big) \neq 0$.
  \begin{enumerate}
    \item[a)] 
        If $E_{2}\left( z_{1},z_{2}\right) >0$, then $P_{z_1,z_2}=0$ 
        has two distinct unimodular roots and two roots off the unit circle;
    \item[b)]
        If $E_{1}\left( z_{1},z_{2}\right) <0$, then $P_{z_1,z_2}=0$ 
        has four distinct unimodular roots;
    \item[c)]
        If $P_{z_1,z_2}=0$ has four distinct unimodular roots, 
        then $E_{2}\left(z_{1},z_{2}\right) <0$;
    \item[d)]
        If $P_{z_1,z_2}=0$ has two distinct unimodular roots and 
        two roots off the unit circle, then $E_{1}\left( z_{1},z_{2}\right) >0$.
  \end{enumerate}
\end{corollary}

\begin{proof}
Denote $s=z_{1}+z_{2}$ and $p=z_{1}z_{2}$.

a) From \eqref{eq:DiscrP}, we have%
\begin{eqnarray} \label{eq:discriP}
  D\left( P_{z_1,z_2}\right) &=
   4\big( \left\vert s\right\vert^{2}-4\left\vert p\right\vert^{2}\big)^{3}
    +108\big( \mbox{Im}\left( s^{2}\overline{p}\right) \big)^{2}.  
\end{eqnarray}
If $E_{2}\left( z_{1},z_{2}\right) >0$, 
then $\left\vert s\right\vert^{2}-4\left\vert p\right\vert^{2}>0$, 
hence $D\left( P_{z_1,z_2}\right) >0$ and by
Proposition \ref{rootnum} \eqref{enu:2} 
$P_{z_1,z_2}=0$ has two distinct unimodular roots 
and two roots off the unit circle.

b) We may write 
\begin{equation*}
  D(P_{z_1,z_2})
    =4\big( \left\vert s\right\vert^{2}-\left\vert p\right\vert^{2}\big) 
     \big( \left\vert s\right\vert^{2}+8\left\vert p\right\vert^{2}\big)^{2}
     -54\big( \left\vert s\right\vert^{4}\left\vert p\right\vert^{2}
     +\mbox{Re}\left( s^{4}\overline{p}^{2}\right) \big) .
\end{equation*}%
Note that $\left\vert s\right\vert^{4}\left\vert p\right\vert^{2}
=\left\vert \overline{s}^{4}p^{2}\right\vert \geq 
-\mbox{Re}\left( s^{4}\overline{p}^{2}\right) $. 
If $E_{1}\left( z_{1},z_{2}\right) <0$, 
then $\left\vert s\right\vert^{2}-\left\vert p\right\vert^{2}<0$, 
therefore $D(P_{z_1,z_2})<0$ and by Proposition \ref{rootnum} \eqref{enu:1}
$P_{z_1,z_2}=0$ has four distinct unimodular roots.

c) Assume that $P_{z_1,z_2}=0$ has four distinct unimodular roots. 
By Proposition \ref{rootnum} \eqref{enu:1} $D(P_{z_1,z_2})<0$, 
but 
\[
  4\big( E_{2}( z_{1},z_{2}) \big)^{2}
   (\left\vert s\right\vert +2\left\vert p\right\vert )^{3}
    =D ( P_{z_1,z_2}) 
     -108\big( \mbox{Im}( s^{2}\overline{p}) \big) ^{2}<0,
\] 
therefore $E_{2}\left( z_{1},z_{2}\right) <0$.

d) Assume that $P_{z_1,z_2}=0$ has two distinct unimodular 
roots and two roots off the unit circle. 
By Proposition \ref{rootnum} \eqref{enu:2} $D(P_{z_1,z_2})>0$, 
but 
\[
  4E_{1}\left(z_{1},z_{2}\right) 
   \left( \left\vert s\right\vert +\left\vert p\right\vert\right) 
   \big( \left\vert s\right\vert^2 +8\left\vert p\right\vert^2 \big)^{2}
   =D \left( P_{z_1,z_2}\right) 
     +54\big( \left\vert s\right\vert^{4}
     \left\vert p\right\vert^{2}+\mbox{Re}( s^{4}\overline{p}^{2})
     \big) >0,
\]
therefore $E_{1}\left( z_{1},z_{2}\right) >0$.
\end{proof}

\begin{remark}
  The domain $E_{1}\left( z_1,z\right) <0$ is the interior of the circle 
  $ E_{1}\left( z_1,z\right) =0$ if $0<\vert z_1\vert <1$, 
  the exterior of the circle 
  $ E_{1}\left(z_1,z\right) =0$ if $\vert z_1\vert >1$
   and a half-plane not containing the
  origin for $\vert z_1\vert =1$. 
  The domain $E_{2}\left( z_1,z\right) >0$ is the exterior
  of the circle $E_{2}\left( z_1,z\right) =0$ if
  $\vert z_1\vert <\frac{1}{2}$, the
  interior of the circle $E_{1}\left( z_{1},z\right) =0$ if 
  $\vert z_1\vert >\frac{1}{2}$
  and a half-plane not containing the origin for $\vert z_1\vert =\frac{1}{2}$ .
\end{remark}

Next we will give geometrical interpretations of Proposition \ref{rootnum}
and \cite[Proposition 4.5]{fhmv} relying on the catacaustic of the unit
circle with radiant point $z_{1}$. 
Recall that the catacaustic of a curve is
the envelope of the family of the reflected rays by that curve, for a light
source at a given point, called the radiant point. In \cite{DG}, Drexler and
Gander studied the equation $Q_{z_{1},z_{2}}(t)=0$ \cite[Eq. (9)]{W} for 
$z_{1},z_{2}\in \mathbb{D}$, assuming without loss of generality that 
$z_{1}=c\in \left( -1,0\right) $. 
They determined the parametrical equations of a curve, called separatrix, 
bounding the regions that correspond to the
cases when the polynomial equation $Q_{z_{1},z_{2}}=0$ has two simple roots,
respectively four simple roots. 
If $z_{2}$ is on the separatrix, then $Q_{z_{1},z_{2}}=0$ has only real roots, 
either one double root and two simple
roots (if $z_{2}$ is a regular point) or a triple root and one simple root
(if $z_{2}$ is a cusp).  
In \cite{DG} the separatrix was represented for six
particular values of the parameter $c$ and highlighted through computer and
optical experiments.  One can see that for a fixed 
$z_{1}=c\in \left(-1,0\right) $ the separatrix from \cite{DG} 
is the catacaustic of the unit circle with the radiant point $z_{1}$, 
denoted in the following by $K_{z_{1}}$. 

Our graphical experiments have shown that the (generalized) circle given by 
$E_{1}\left( z_{1},z\right) :=\left\vert z_{1}+z\right\vert
 -\left\vert z_{1}z\right\vert =0$ is tangent to the catacaustic 
$K_{z_{1}}$ and that the (generalized) circle given by 
$E_{2}\left(z_{1},z\right) :=\left\vert z_{1}+z\right\vert -2\left\vert
  z_{1}z\right\vert =0$ intersects $K_{z_{1}}$ exactly at the cusps of 
$K_{z_{1}}$ (see, Figure \ref{fig:caustic}).
In the following we will explain the results of these
experiments using the connection between the discriminant of the polynomial 
$P_{z_{1},z_{2}}$ and the resultant 
$  \mbox{{\rm resul}}_u     
  \left(P_{z_{1},z_{2}},P_{z_{1},z_{2}}^{\prime }\right) $. 

The equation of the reflected ray for the incident ray passing through $%
z_{1} $, with the reflection point $u=e^{i\varphi }$, is 
\begin{equation*}
  \left\vert 
     \begin{array}{ccc}
       z & \overline{z} & 1 \\ 
       u & 1/u & 1 \\ 
       \overline{z_{1}}u^{2} & z_{1}/u^{2} & 1%
     \end{array}%
  \right\vert =0,
\end{equation*}%
which is equivalent to 
\begin{equation*}
  F_{z_{1}}(z,\overline{z},u):=(u-z_{1})z+u^{3}\left( \overline{z_{1}}%
  u-1\right) \overline{z}-u(\overline{z_{1}}u^{2}-z_{1})=0\text{.}
\end{equation*}

Note that $F_{z_{1}}(z,\overline{z},u)=P_{z_{1},z}\left( u\right) $ for all 
$u,z_{1},z\in \mathbb{C}$.

The envelope of the family of curves $F_{z_{1}}(z,\overline{z},u)=0$, $u\in
\partial \mathbb{D}$ is the catacaustic of the circle, with radiant point $%
z_{1}$, which will be denoted by $K_{z_{1}}.$The standard equations of 
$K_{z_{1}}$ are 
\begin{equation}\label{eq:Env2}
  \left\{ 
    \begin{array}{c}
      F_{z_{1}}(z,\overline{z},u)=0 \\ 
      \dfrac{\partial }{\partial u}F_{z_{1}}(z,\overline{z},u)=0%
    \end{array}%
  \right. \text{, }u\in \partial \mathbb{D}\text{. }  
\end{equation}

Solving \eqref{eq:Env2} as a linear system with respect to $z$ 
and $\overline{z} $ it follows that: 
\begin{equation*}
  z=\frac{u(\overline{z_{1}}^{2}u^{3}-3\left\vert z_{1}\right\vert^{2}u+2z_{1})}
         {3\overline{z_{1}}u^{2}-2(2\left\vert z_{1}\right\vert^{2}+1)u+3z_{1}}
    \text{ \ and \ }
  \overline{z}=\frac{2\overline{z_{1}}u^{3}-3\left\vert z_{1}\right\vert^{2}u^{2}
     +z_{1}^{2}}{u^{2}\left( 3\overline{z_{1}}u^{2}
   -2(2\left\vert z_{1}\right\vert^{2}+1)u+3z_{1}\right) }.
\end{equation*}%
The parametric equations of the catacaustic, obtained from the above
equalities for $u=e^{i\varphi }$ are: 
\begin{equation*}
  \left\{ 
    \begin{array}{c}
     x=\dfrac{1}{2}
       \dfrac{\mbox{Re}( \overline{z_{1}}^{2}u^{3}-3\left\vert
             z_{1}\right\vert^{2}u+2z_{1}) }
            {3\mbox{Re}\left( \overline{z_{1}}u\right) 
            -(2\left\vert z_{1}\right\vert^{2}+1)
           \rule[-5pt]{0pt}{10pt}} \\ 
     y=\dfrac{1}{2}
       \dfrac{\mbox{Im}( \overline{z_{1}}^{2}u^{3}-3\left\vert 
             z_{1}\right\vert^{2}u+2z_{1}) }
           {3\mbox{Re}\left( \overline{z_{1}} u\right) 
           -(2\left\vert z_{1}\right\vert^{2}+1)}
    \end{array}
  \right. \text{, \ where }u=e^{i\varphi }\text{ \ and \ } z=x+iy.
\end{equation*}

Due to rotation invariance, we may assume that $z_{1}$ is a positive number.
In the case when $z_{1}=c$ is a positive real number, we get the well-known
parametric equations of the catacaustic of the circle with a radiant point
on the positive real axis \cite{Bromwich}: 
\begin{equation*}
  \left\{ 
    \begin{array}{l}
       x=\dfrac{c}{2}
         \dfrac{c\cos 3\varphi -3c\cos \varphi +2}{3c\cos \varphi-(2c^{2}+1)
             \rule[-5pt]{0pt}{10pt}}\\ 
       y=\dfrac{c}{2}
         \dfrac{c\sin 3\varphi -3c\sin \varphi }{3c\cos \varphi -(2c^{2}+1)}
    \end{array}%
  \right. .
\end{equation*}

The system \eqref{eq:Env2} is equivalent to 
$ \mbox{{\rm resul}}_{u} 
\left( F_{z_{1}}(z,\overline{z},u),
\frac{\partial }{\partial u}F_{z_{1}}(z,\overline{z},u)\right) =0$, 
which is an implicit equation of the catacaustic $K_{z_{1}}.$

Since $F_{z_{1}}(z,\overline{z},u)=P_{z_{1},z}\left( u\right) $ for all 
$u,z_{1},z\in \mathbb{C}$, using the connection between the discriminant of a
polynomial $f(x) $ and the resultant 
$ \mbox{{\rm resul}}_x  
 \left( f,f^{\prime }\right) $
we see that 
\begin{equation*}
  D\left( P_{z_{1},z}\right) 
    =\frac{1}{\overline{z_{1}z}}\;
          \mbox{{\rm resul}}_{u} 
           \Big(F_{z_{1}}(z,\overline{z},u),
               \frac{\partial }{\partial u}F_{z_{1}}(z,\overline{z},u)\Big) .
\end{equation*}
It follows that a simpler implicit equation of the catacaustic $K_{z_{1}}$
is 
\begin{equation*}
 \left( K_{z_{1}}\right) \text{: }D\left( P_{z_{1},z}\right) =0.
\end{equation*}
Using the formula \eqref{eq:discriP} with $s=z_{1}+z$ and $p=z_{1}z$ we get 
\begin{equation}
  D\left( P_{z_{1},z}\right) =4\Big( \left\vert z+z_{1}\right\vert
   ^{2}-4\left\vert z_{1}\right\vert^{2}\left\vert z\right\vert^{2}\Big)^{3}
  +108\Big( \mbox{Im}\big( \left( z+z_{1}\right) ^{2}\overline{z_{1}}%
  \overline{z}\big) \Big) ^{2}.  \label{Discrivar}
\end{equation}
Therefore, an implicit equation of the catacaustic of the circle with the
radiant point $z_{1}$ is 
\begin{equation*}
  4\Big( \left\vert z+z_{1}\right\vert^{2}-4\left\vert z_{1}\right\vert
   ^{2}\left\vert z\right\vert^{2}\Big)^{3}
  +108\Big( \mbox{Im}\big(
   \left( z+z_{1}\right) ^{2}\overline{z_{1}}\overline{z}\big) \Big) ^{2}=0.
\end{equation*}
Note that we may write 
\begin{eqnarray*}
  D\left( P_{z_{1},z}\right) 
   & = 4\big( \left\vert z+z_{1}\right\vert^{2}
       -\left\vert z_{1}\right\vert^{2}\left\vert z\right\vert^{2}\big)
        \big( \left\vert z+z_{1}\right\vert^{2}
       +8\left\vert z_{1}\right\vert^{2}\left\vert z\right\vert^{2}\big)^{2}\\
   &\quad   -54\Big( \left\vert z+z_{1}\right\vert^{4}\left\vert 
       z_{1}\right\vert^{2}\left\vert z\right\vert^{2}
       +\mbox{Re}\big( \left( z+z_{1}\right) ^{4}\overline{z_{1}}^{2}
      \overline{z}^{2}\big) \Big) .
\end{eqnarray*}

In conclusion, $z_{2}\in K_{z_{1}}$ if and only if $D\left(
P_{z_{1},z_{2}}\right) =0$. By Proposition \ref{rootnum} \eqref{enu:3}, 
if $D\left(P_{z_{1},z_{2}}\right) =0$, then all the roots $u$ of $P_{z_{1},z_{2}}=0$ 
are unimodular and $P_{z_{1},z_{2}}=0$ has either a double root and two simple
roots, or a triple root and one simple root. The latter case occurs if and
only if $z_{2}$ is a cusp of the catacaustic of the circle with radiant
point $z_{1}$, as shown in \cite{B} in the case where $z_{1}=c$ a positive
number, see also \cite[Lemma 4.3]{fhmv}.

\begin{figure}
  \centerline{
    {\includegraphics[width=0.4\linewidth]{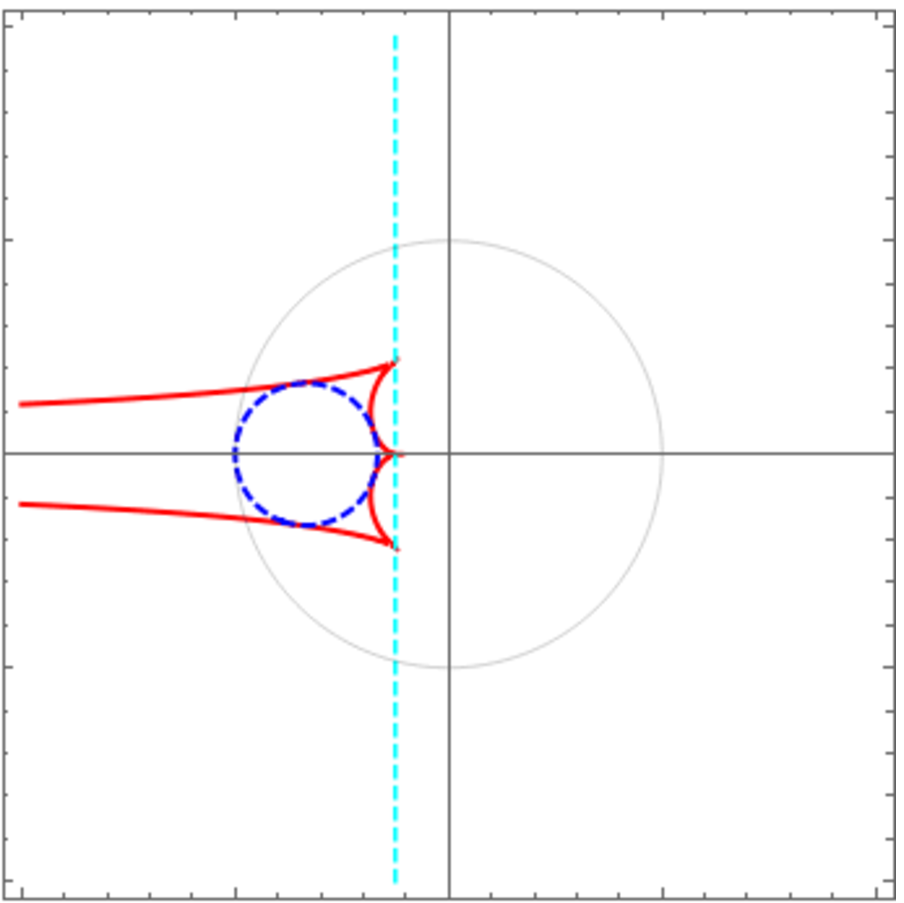}}
      \qquad
    {\includegraphics[width=0.4\linewidth]{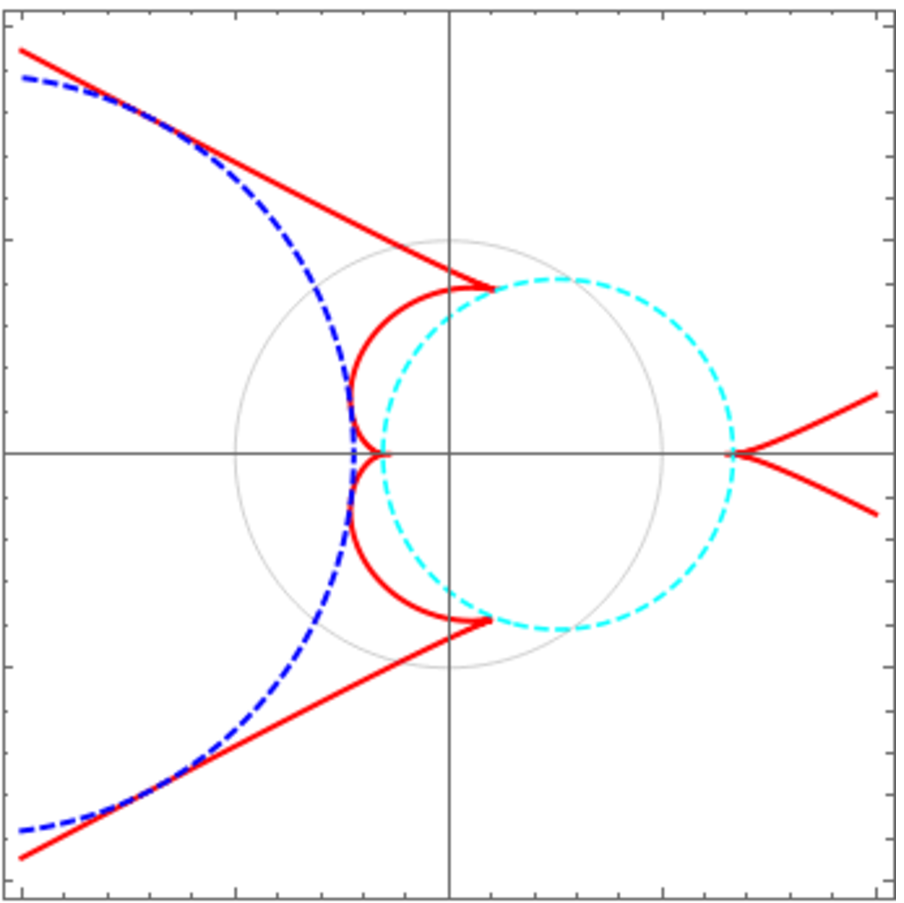}}
      }
  \caption{The catacaustics of the unit circle with radiant points 
           $ c=0.5 $ (left) and $ c=0.8 $ (right).
           The thick (red) curves indicate the catacaustics.
           The thick (blue) and thin (cyan) dotted circles
           represent $ E_1(c,z)=0 $ and $ E_2(c,z)=0 $, respectively.}
  \label{fig:caustic}
\end{figure}

If $z_{2}$ belongs to a connected component of the complement of the
catacaustic $D\left( P_{z_{1},z}\right) =0$ we have either $D\left(
P_{z_{1},z_{2}}\right) >0$ (hence $P_{z_{1},z_{2}}=0$ has two simple
unimodular roots and two distinct roots off the unit circle) or $D\left(
P_{z_{1},z_{2}}\right) <0$ (hence $P_{z_{1},z_{2}}=0$ has four simple
unimodular roots).

Assuming that $z_{1}\neq 0$, using Corollary \ref{diskroots} we easily
identify the regions $D\left( z_{1},z\right) >0$ and $D\left( z_{1},z\right)
<0$, since 
\begin{eqnarray*}
  0 &\in &\left\{ z\in \mathbb{C}:E_{2}\left( z_{1},z\right) >0\right\}
   \subset \left\{ z\in \mathbb{C}:D\left( z_{1},z\right) <0\right\} 
  \text{ and} \\
  \left( -z_{1}\right) &\in &\left\{ z\in \mathbb{C}:E_{2}\left(
   z_{1},z\right) <0\right\} \subset \left\{ z\in \mathbb{C}:D\left(
   z_{1},z\right) >0\right\} .
\end{eqnarray*}

\begin{remark}
  If $\left\vert z_{1}\right\vert >1$, then the catacaustic $K_{z_{1}}$ is a
  Jordan curve contained in the closed unit disk. If $\left\vert
  z_{1}\right\vert >1$ and \thinspace $\left\vert z_{2}\right\vert >1$ it
  follows that $D\left( P_{z_{1},z_{2}}\right) <0$, hence the equation $%
  P_{z_{1},z_{2}}\left( u\right) =0$ has four simple unimodular roots, as it
  is proved geometrically in \cite[Proposition 3.8 (i)]{fhmv}.
\end{remark}

\section{{Alhazen's problem and the circle of Apollonius}}
\label{sect:apollo}
The circle of Apollonius is the locus of points that 
have a constant ratio of distances from two given points.
Figure \ref{fig:apollo} illustrates the two geometric ideas
to find the PA-point on the unit circle,
the ellipse method (on the left)
and the angle bisection property of the circle of Apollonius (on the right).

\begin{figure}
  \centerline{
     \fbox{\includegraphics[width=0.4\linewidth]{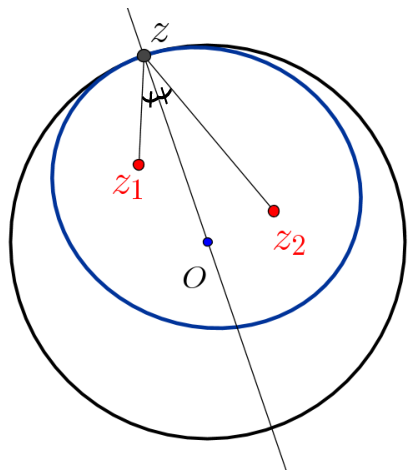}}
      \qquad
     \fbox{\includegraphics[width=0.4\linewidth]{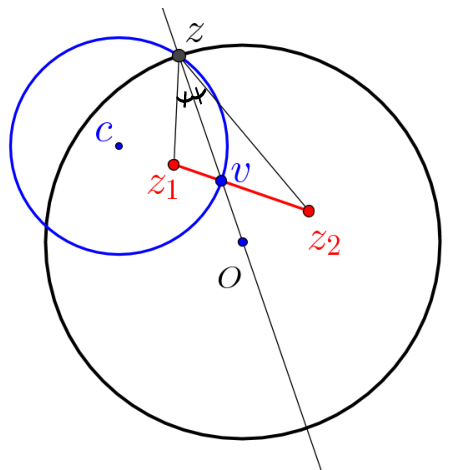}}
      }
  \caption{Solution using an ellipse (left). 
           Solution using the circle of Apollonius (right).}
  \label{fig:apollo}
\end{figure}

\begin{lemma}\label{lem:apollo}
  Assume that the segment $ [z_1,z_2] $ is either in $ \mathbb{D} $,
  or in $ \mathbb{C}\setminus\overline{\mathbb{D}} $,
  and that $ 0,\, z_1,\, z_2 $ are not collinear.

  Then, $ z\in\partial\mathbb{D} $ is a PA-point if and only if there exists
  some $ t \in(0,1) $ such that  
  $ z=\pm\big(tz_2+(1-t)z_1\big)/\vert tz_2+(1-t)z_1\vert  $ 
  and $ z $ belongs to
  the circle of Apollonius with respect to the two points $ z_1,z_2 $ 
  and the ratio $ t $.
\end{lemma}

\begin{proof}
Sufficiency follows from the angle bisector property for 
the circle of Apollonius.

{\it Necessity}:
Let $ z \in\partial\mathbb{D} $ be a PA-point.
The line $ L_z $ joining $ z $ to the origin bisects the angle
$ \angle(z_1,z,z_2) $.
If $ \vert z_1\vert =\vert z_2\vert  $ and $ L_z $ 
is the perpendicular bisector of the segment
$ [z_1,z_2] $, the claim follow with $ t=1/2 $.
This corresponds to the case in which the circle of Apollonius 
is a straight line.

In the remaining case, the intersection between $ L_z $ and the segment 
$ [z_1,z_2] $ is a point $ v=tz_2+(1-t)z_1 $ for some 
$ t\in(0,1)\setminus\{1/2\} $.
Let  $ u $ be the harmonic conjugate of $ v $ with respect to 
$ z_1 $ and $ z_2 $.
Then $ z $ belongs to the circle of diameter $ [u,v] $, 
which is the circle of Apollonius with respect to the two points $ z_1,z_2 $ 
and the ratio $ t $.
\end{proof}

\begin{proposition}\label{prop:apollo}
  Assume that the segment $ [z_1,z_2] $ is either in $ \mathbb{D} $,
  or in $ \mathbb{C}\setminus\overline{\mathbb{D}} $,
  and that $ 0,\, z_1,\, z_2 $ are not collinear.
  Then the point $ z \in\partial\mathbb{D} $ is a PA-point 
  if and only if $ z $ is written as
  $ z=\pm\big(tz_2+(1-t)z_1\big)/\vert tz_2+(1-t)z_1\vert  $ for some 
  $ t\in(0,1) $ such that
   \begin{equation}\label{eq:apollo}
     (1-2t)^2\Big((1-t)^2\vert z_1\vert^2+t^2\vert z_2\vert^2+
        2t(1-t)\mbox{\rm Re}\,(z_1\overline{z_2})\Big)
      =\Big((1-t)^2\vert z_1\vert^2-t^2\vert z_2\vert^2\Big)^2.
   \end{equation}
\end{proposition}

\begin{proof}
By Lemma \ref{lem:apollo}, $ z\in\partial\mathbb{D} $ is a PA-point
if and only if there exists some $ t\in(0,1) $
such that $ z=\pm\big(tz_2+(1-t)z_1\big)/\vert tz_2+(1-t)z_1\vert  $ 
belongs to the circle of Apollonius with respect to the two points $ z_1,z_2 $ 
and the ratio $ t $.

For $ \vert z_1\vert =\vert z_2\vert  $ and $ t=1/2 $, 
\eqref{eq:apollo} holds, and $ z $ is a PA-point.

We will consider other cases.
Let $ v $ be the internal division point of the segment $ [z_1,z_2] $ 
in the ratio $ t:(1-t) $ and $ u $ the external division point of 
$ [z_1,z_2] $ in the same ratio, i.e.,
\[
  v=tz_2+(1-t)z_1,\qquad  u=\frac{-tz_2+(1-t)z_1}{1-2t} .
\]
The corresponding circle of Apollonius is given by 
$ \vert \zeta-(u+v)/2\vert =\vert (u+v)/2c-v\vert  $, which is equivalent to
\begin{equation}\label{eq:uv-apollo}
  \vert z\vert^2-\mbox{\rm Re}\,\big((u+v)\overline{z}\big)
      +\mbox{\rm Re}\,(\overline{u}v)=0.
\end{equation}
By \eqref{eq:uv-apollo}, $ z=\pm v/\vert v\vert  $ is on the corresponding 
circle of Apollonius if and only if
\[
    \mbox{\rm Re}\,(\overline{u}v)
      \Big( 1\mp \frac{1}{\vert v\vert }\Big)=\pm \vert v\vert -1.
\]
As the segment $ [z_1,z_2] $ does not intersect the unit circle,
we have $ \vert v\vert -1\neq 0 $.
Then $ z=\pm v/\vert v\vert  $ is on the circle of Apollonius with diameter 
$ \vert u-v\vert  $
if and only if $ \mbox{Re}\,(\overline{u}v)=\pm \vert v\vert  $, i.e.
\[
   \big(\mbox{Re}\,(\overline{u}v)\big)^2=\vert v\vert^2.
\]
Inserting here the formulas of $ u $ and $ v $, we get \eqref{eq:apollo}.
\end{proof}

\begin{remark}
  The equation \eqref{eq:apollo} can also be written as follows.
  \begin{align*}
  T(z_1,z_2) = &\big((\vert z_1\vert^2-\vert z_2\vert^2)^2
             -4\vert z_1-z_2\vert^2\big)t^4 \\
      &   -4\big(\vert z_1\vert^2(\vert z_1\vert^2-\vert z_2\vert^2)
          -2\vert z_1-z_2\vert^2-\vert z_1\vert^2+\vert z_2\vert^2\big)t^3\\
      & +\big(2\vert z_1\vert^2(3\vert z_1\vert^2-\vert z_2\vert^2)
         -8\vert z_1\vert^2+4\vert z_2\vert^2
          -5\vert z_1-z_2\vert^2\big)t^2\\
      &  -\big(\vert z_1\vert^2(4\vert z_1\vert^2-5)+\vert z_2\vert^2
          -\vert z_1-z_2\vert^2\big)t
           +\vert z_1\vert^2(\vert z_1\vert^2-1)=0.
  \end{align*}
  Therefore, if $ z_1,z_2\in\mathbb{D} $,
  the equation $ T(z_1,z_2)=0 $ has at least two real roots $ t_1,t_2 $
  satisfying $ 0<t_1\leq 1/2 \leq t_2<1 $ since the following hold
  \begin{align*}
      & T(z_1,z_2)\vert _{t=0}=z_1\overline{z_1}(z_1\overline{z_1}-1)<0,\quad
        T(z_1,z_2)\vert _{t=1}=z_2\overline{z_2}(z_2\overline{z_2}-1)<0,\\
      & \mbox{and}\quad
        T(z_1,z_2)\vert _{t=\frac12}
           =\frac{1}{16}(z_1\overline{z_1}-z_2\overline{z_2})^2\geq0.
  \end{align*}
{ 
  So, the PA-point that attains the supremum in the definition of the triangular ratio metric
  $ s_{\mathbb{D}}(z_1,z_2) $  is on the shorter of the two arcs connecting 
   $ e^{i\arg (z_1)} $ and $ e^{i\arg (z_2)} $.
}
\end{remark}

\section{Alhazen's problem on a conic domain}\label{sect:conic}

In this section, we consider Alhazen's problem for quadric surfaces.
Since the section of a quadric surface (conicoid) by each plane 
is a quadratic curve (conic), 
we will consider this problem in a planar domain whose boundary is a conic.

\begin{problem}\label{prob:conic}
  For two points $ z_1,\, z_2\in\mathbb{C} $ and a conic domain $ D $,
  find the PA-points on $ \partial D$.
\end{problem}
Since it is difficult to extend the solution method
using the circle of Apollonius,
we will apply the solution method using ellipses to this problem.

Here, we consider similarity transformations 
of the form $ A(z)=\alpha z+\beta $, where $ \alpha,\, \beta \in\mathbb{C}$.
Then $ A $ maps each subdomain $ D $ of $ \mathbb{C} $
onto $ A(D) $ which is a translated, rotated, and rescaled version
of $ D $.
The similarity transformation
\[
     A(z)=\frac{2}{z_1-z_2}z-\frac{z_1+z_2}{z_1-z_2},
\]
sends the points $ z_1 $  and $ z_2 $ to $ 1 $ and $ -1 $, respectively,
and  maps conics to conics.
So, instead of Problem \ref{prob:conic},
we just need to solve the following problem.

\begin{problem}\label{prob:conic2}
  For two points $ z_1,\, z_2\in\mathbb{C} $ and a conic domain $ D $,
  let $ C $ be the boundary of $ A(D) $, where
  $ A(z)=2/(z_1-z_2)z-(z_1+z_2)/(z_1-z_2) $.
  Then, find $ u\in C $ such that
  \[
       \vert \angle(-1,u,0)\vert =\vert \angle(0,u,1)\vert . 
  \]
\end{problem}

Let $ C(=\partial A(D)) $  be a conic given by
\begin{equation}\label{eq:C}
   C\ : \ c(z)=\overline{a}z^2+pz\overline{z}+a\overline{z}^2
        +\overline{b}z+b\overline{z}+q=0 \qquad 
        (a,b\in\mathbb{C}, \ p,q \in\mathbb{R}),
\end{equation}
and $ E $ the ellipse with foci $ 1 $  and $ -1 $
\begin{equation}\label{eq:E}
   E\ :\    \vert z-1\vert +\vert z+1\vert =r \quad (r>2). 
\end{equation}
The ellipse $ E $ is also expressed as
\begin{equation}\label{eq:ez}
   e(z)=z^2+(2-2r_2)z\overline{z}+\overline{z}^2+r_2^2-2r_2=0,
\end{equation}
where $ r_2:=r^2/2 \,(>2) $.

Note that \eqref{eq:ez} is obtained by squaring both sides of 
the equation \eqref{eq:E}, so it includes the hyperbola 
$ \big\vert \vert z-1\vert -\vert z+1\vert \big\vert =r $ 
as well as the ellipse $ E $.

\medskip

From now on, we assume that $ C $ is a non-degenerate conic,
i.e. $ C $ is an ellipse, a hyperbola, or a parabola.
A conic \eqref{eq:C} can be classified as follows.
For $ p^2-4a\overline{a}<0 $, $C$ represents 
a hyperbola or its degenerate form;
for $ p^2-4a\overline{a}>0 $, $C$ represents 
an ellipse or its degenerate form;
and for $ p^2-4a\overline{a}=0 $, $C$ represents 
a parabola or its degenerate form.
For a more detailed classification of conics $C$, see for example
\cite[Lemma 3]{bla-circum}.

The following result is an extension of Lemma \ref{lem:pa-eq}.
Since the boundary of the domain is extended from the unit circle 
to a conic, the technique of the proof of Lemma \ref{lem:pa-eq},
i.e. Theorem 1.1 in \cite{fhmv}, is not available.
However, it is difficult to calculate it manually and directly.
Here we use the Risa/Asir, the symbolic computation system,
to {process}  the equation.

\begin{theorem}\label{lem:tan}
  {      
   Let $ C $  be a conic given by
   $ c(z)=\overline{a}z^2+pz\overline{z}+a\overline{z}^2
        +\overline{b}z+b\overline{z}+q=0 \
     (a,b\in\mathbb{C}, \ p,q \in\mathbb{R}),
   $
   and $ E $ the ellipse given by
   $ e(z)=z^2+(2-2r_2)z\overline{z}+\overline{z}^2+r_2^2-2r_2=0 $.
  }
  Suppose $ E $ does not coincide with $ C $ for some $ r $, 
  and has a point of tangency on $ C $ for some $ r $.
  The point of tangency is given by the solution of
  $ F_4=W_6z^6+W_5z^5+W_4z^4+W_3z^3+W_2z^2+W_1z+W_0=0 $, where
\begin{align*}
 W_6&=4\overline{a}(\overline{b}bp-\overline{b}^2a
        -\overline{a}b^2)(p^2-4\overline{a}a),\\
%
  W_5&= -2\Big[
      bp^5-\overline{b}(a+\overline{a})p^4-b(4\overline{a}q+8\overline{a}a
     +\overline{b}^2)p^3
     +\overline{b}\big( 8\overline{a}aq+8\overline{a}a^2
     +(8\overline{a}^2+\overline{b}^2)a \\*
  & 
     -3\overline{a}b^2\big)p^2
     +4\overline{a}b(4\overline{a}aq+4\overline{a}a^2+4\overline{b}^2a
     +\overline{a}b^2)p-4a\overline{a}\overline{b}
     \big( 8\overline{a}aq
     +4\overline{a}a^2\\*
  &
     +(4\overline{a}^2+3\overline{b}^2)a+3\overline{a}b^2
     \big)\Big],\\
%
 W_4&=-\Big[p^6-\big(
     (4a+4\overline{a})q+a^2+10\overline{a}a-9b^2+\overline{a}^2
     +\overline{b}^2\big)p^4-2b\overline{b}(2q+11a+\overline{a})p^3 \\*
 &  
     +\Big(16\overline{a}aq^2+4\big(
      8\overline{a}a^2+(8\overline{a}^2+2\overline{b}^2)a
       -5\overline{a}b^2\big)q
     +8\overline{a}a^3+(32\overline{a}^2+14\overline{b}^2)a^2 \\*
 &
     +(-42\overline{a}b^2+8\overline{a}^3+18\overline{b}^2\overline{a})a
     +(-6\overline{a}^2-5\overline{b}^2)b^2\Big)p^2 
     +2b\overline{b}\big(
     48\overline{a}aq+44\overline{a}a^2 \\*
 & 
     +(4\overline{a}^2+7\overline{b}^2)a
     +\overline{a}b^2\big)p 
    -64\overline{a}^2a^2q^2-16a^2\overline{a}(4\overline{a}a+4\overline{a}^2
    +7\overline{b}^2)q
    -16\overline{a}^2a^4-(32\overline{a}^3 \\*
 & 
     +56\overline{b}^2\overline{a})a^3 
     +(24\overline{a}^2b^2-16\overline{a}^4
    -56\overline{b}^2\overline{a}^2-9\overline{b}^4)a^2+(24\overline{a}^3
    -6\overline{b}^2\overline{a})b^2a+3\overline{a}^2b^4\Big],\\
%
 W_3&= -2\Big[
     2bp^5-\overline{b}(q+4a)p^4-b\big((14a+2\overline{a})q+2a^2
     +12\overline{a}a
     -8b^2+2\overline{a}^2+\overline{b}^2\big)p^3 \\*
 & 
    +\overline{b}\big(4aq^2+(20a^2
    +20\overline{a}a-7b^2)q+4a^3+24\overline{a}a^2
    +(-24b^2+4\overline{a}^2+3\overline{b}^2)a-\overline{a}b^2\big)p^2 \\*
 &
    +b\Big(32\overline{a}aq^2+\big(56\overline{a}a^2+(8\overline{a}^2
    +26\overline{b}^2)a-6\overline{a}b^2\big)q+8\overline{a}a^3
    +(16\overline{a}^2+26\overline{b}^2)a^2\\*
 & 
    +(-18\overline{a}b^2
     +8\overline{a}^3+2\overline{b}^2\overline{a})a+(-6\overline{a}^2
    -\overline{b}^2)b^2\Big)p
    -\overline{b}\Big(80\overline{a}a^2q^2+4a\big(20\overline{a}a^2 \\*
 &
    +(16\overline{a}^2+6\overline{b}^2)a   
    -3\overline{a}b^2\big)q+16\overline{a}a^4
    +(32\overline{a}^2+12\overline{b}^2)a^3 
    +(-28\overline{a}b^2+16\overline{a}^3+16\overline{b}^2\overline{a})a^2 \\*
 & 
    +(-24\overline{a}^2-\overline{b}^2)b^2a-\overline{a}b^4\Big)\Big], \\
%
 W_2&= 2\Big[\big( (3a-\overline{a})q-3b^2 \big)p^4
     +b\overline{b}(3q+9a+\overline{a})p^3
     +\Big(-2(6a^2+2\overline{a}a-b^2)q^2 -\big(4a^3 \\*
 & 
    +16\overline{a}a^2
    +(-27b^2-4\overline{a}^2+9\overline{b}^2)a
    +4\overline{a}b^2\big)q+3(b^2-3\overline{b}^2)a^2
    +(9\overline{a}b^2-\overline{b}^2\overline{a})a-7b^4\\*
 &  
   +2\overline{a}^2b^2\Big)p^2
   -b\overline{b}\big(28aq^2+(60a^2-5b^2)q+8a^3+20\overline{a}a^2-(19b^2
   -12\overline{a}^2)a-5\overline{a}b^2\big)p \\*
 & 
   +32\overline{a}a^2q^3
   +4a\big(12\overline{a}a^2+(4\overline{a}^2+11\overline{b}^2)a
   -3\overline{a}b^2\big)q^2
   +\big(16\overline{a}a^4+(16\overline{a}^2
   +44\overline{b}^2)a^3\\*
 & 
   +(-32\overline{a}b^2+40\overline{b}^2\overline{a})a^2
   -(12\overline{a}^2+9\overline{b}^2)b^2a+\overline{a}b^4\big)q
   +8\overline{b}^2a^4+(-4\overline{a}b^2+20\overline{b}^2\overline{a})a^3 \\*
 & 
   +\big(-(4\overline{a}^2+15\overline{b}^2)b^2+12\overline{b}^2\overline{a}^2
   +3\overline{b}^4\big)a^2
   +(5\overline{a}b^4-15\overline{b}^2\overline{a}b^2)a
   +2\overline{a}^2b^4\Big],\\
%
W_1&=2\Big[2b\big(q^2+(3a-\overline{a})q-b^2\big)p^3
   -\overline{b}\big(8aq^2+(12a^2
   -4\overline{a}a+b^2)q-6b^2a-2\overline{a}b^2\big)p^2\\*
 & 
   -b\Big(8aq^3+(32a^2
   -8\overline{a}a-2b^2)q^2+\big(8a^3+8\overline{a}a^2
   +(-20b^2-2\overline{b}^2)a
   +2\overline{a}b^2\big)q \\*
 & 
   +(-2b^2+6\overline{b}^2)a^2+(-2\overline{a}b^2
   +6\overline{b}^2\overline{a})a+3b^4-\overline{b}^2b^2\Big)p 
   +\overline{b}\Big(32a^2q^3\\*
 & 
   +4a(12a^2+4\overline{a}a-3b^2)q^2+\big(16a^4
   +16\overline{a}a^3+(-32b^2+8\overline{b}^2)a^2
   -12\overline{a}b^2a+b^4\big)q\\*
 & 
   +(-4b^2+4\overline{b}^2)a^3+(-4\overline{a}b^2
   +4\overline{b}^2\overline{a})a^2+(5b^4-3\overline{b}^2b^2)a
   +2\overline{a}b^4\Big)\Big],\\
%
W_0 & =q^2p^4-2\overline{b}bqp^3
   +\big(-8aq^3+(-8a^2+2b^2)q^2+(6b^2+2\overline{b}^2)aq
   -b^4+\overline{b}^2b^2\big)p^2\\*
 &
   +2b\overline{b}\big(4aq^2+(-4a^2-b^2)q+(b^2
   -\overline{b}^2)a\big)p
   +16a^2q^4+8a(4a^2-b^2)q^3\\*
 & 
   +\big(16a^4+(-32b^2+8\overline{b}^2)a^2+b^4\big)q^2 
   -2a\big((4b^2-4\overline{b}^2)a^2-5b^4+3\overline{b}^2b^2\big)q \\*
 & 
   +(b-\overline{b})(b+\overline{b})\big((b^2-\overline{b}^2)a^2-b^4\big).
\end{align*}
\end{theorem}
\begin{proof}
Note that, if $ a=0 $, the conic $ C $ is a circle
and we can apply Lemma \ref{lem:pa-eq}.
Therefore, we assume that $ a\neq 0 $.

From the assumption of Theorem \ref{lem:tan}, 
there is an intersection point of $ E $ and $ C $.
Eliminating $ \overline{z} $ from $ c(z)=0 $ and $ e(z)=0 $,
we have the following quartic equation of $ z $ variable, 
\begin{align*}
   S(z) = & \big(4\overline{a}ar_2^2+2((a+\overline{a})p-4\overline{a}a)r_2
           +(p-a-\overline{a})^2\big)z^4 \\
         & +\big(4\overline{b}ar_2^2+2(\overline{b}p+(b-4\overline{b})a
           +\overline{a}b)r_2+2(b-\overline{b})(p-a-\overline{a})\big)z^3 \\
         &  +\big(2apr_2^3+(p^2-6ap+4aq+2a^2-2\overline{a}a)r_2^2 
             -2(p^2-(q+2a)p+4aq+2a^2\\
         & \quad -2\overline{a}a-\overline{b}b)r_2
            -2qp+(2a+2\overline{a})q+b^2-2\overline{b}b+\overline{b}^2\big)z^2\\
         &  +\big(2bar_2^3+2(bp+(-3b-\overline{b})a)r_2^2-2(2bp-bq+(-2b
           -2\overline{b})a)r_2-2(b-\overline{b})q\big)z\\
         & +a^2r_2^4-4a^2r_2^3+(-2aq+4a^2+b^2)r_2^2+(4aq-2b^2)r_2+q^2=0.
\end{align*}
As there is a point of tangency of two curves $ C $ and $ E $,
the equation $ S(z)=0 $ must have multiple roots.
This condition is equivalent to the requirement that the system of equations
\begin{equation}\label{eq:common}
     S(z)=0, \quad \mathrm{and} \quad S'(z)=0 
\end{equation}
has a common root.

Here, we remark that the leading coefficient of $ S(z)=0 $ as $ z $  variable
is not constant zero. 
In fact, if 
\[ 
   4\overline{a}ar_2^2+2\big((a+\overline{a})p-4\overline{a}a\big)r_2
   +(p-a-\overline{a})^2  \equiv 0 
\] 
then $ a=p=0  $ holds, and
$ C:\ c(z)=\overline{b}z+b\overline{z}+q=0 $ degenerates to a line.
Moreover, the leading term $ a^2r_2^4 $ of $ S(z)=0 $ as $ r_2 $ variable
does not vanish by the assumption.

Now, eliminating $ r_2 $ from \eqref{eq:common} by
calculating
\begin{equation}\label{eq:resul}
   \mbox{{\rm resul}}_{r_2} 
     (S(z),S'(z))=0, 
\end{equation}
we have $ a^2 F_1F_2F_3^2F_4=0 $, where
\begin{align*}
    F_1(z)&=(p-a-\overline{a})z^2+(b-\overline{b})z-q,\\
    F_2(z)&=(p+a+\overline{a})z^2+(b+\overline{b})z+q,\\
    F_3(z)&=(4a^2-4\overline{a}a)z^4-4\overline{b}az^3+(p^2-4aq-4a^2)z^2
             +2bpz+b^2.
\end{align*}
Here, we use Risa/Asir, a symbolic computation system, for computing
the resultant in \eqref{eq:resul}. 
(See, for example, \cite{cox2} for details on the relationship 
between the resultant and the solution of a system of equations.)

We need to examine the properties of factors $ F_1,\cdots,F_4 $.
\begin{itemize}
  \item The factor $ a\neq0 $ from the assumption.
        (If $ a=0 $, $ C $ is a circle.)
  \item The equation $ F_1=0 $ is obtained from substituting
        $ -z $ for $ \overline{z} $ in $ c(z)=0 $.
        Therefore $ F_1=0 $ gives 
        the intersection points of $ C $ and the imaginary axis.
  \item The equation $ F_2=0 $ is obtained from substituting 
        $ z $ for $ \overline{z} $ in  $ c(z)=0 $.
        Therefore $ F_2=0 $ gives the intersection points of
        $ C $ and the real axis.
  \item The solutions of $ F_3=0 $ give the condition that there exist 
        multiple roots for $ r_2 $.
        In fact, we have the following,
        \begin{align*}
           & \mbox{{\rm resul}}_{r_2}  
                 \Big(S,\, \frac{\partial}{\partial r_2}S\Big) \\
           & \quad 
               =16a^4(z-1)^2(z+1)^2(\overline{a}z^2+(-p+\overline{b})z+q+a-b)\\
           & \qquad
              \times (\overline{a}z^2+(p+\overline{b})z+q+a+b)
                  ((p^2-4\overline{a}a)z^2+(2bp-4\overline{b}a)z-4aq+b^2)^2\\
           & \qquad
              \times((4a^2-4\overline{a}a)z^4-4\overline{b}az^3+(p^2-4aq-4a^2)z^2
                    +2bpz+b^2)^2.
        \end{align*}
\end{itemize}
Therefore, the equality $ aF_1F_2F_3=0 $ does not give the condition that
the equation $ S=0 $ has multiple roots.
Hence, the equation $ F_4=0 $ gives the condition 
that the equation $ S=0 $ has multiple roots and includes 
a point of tangency as its solution.
\end{proof}

\begin{remark}
  In the above proof, the conic $ C $ is assumed not to be the circle.
  In the case that $ a=0 $ and $ p\neq0 $, $C$ represents a circle. 
  Moreover, we can set $ p=1 $ without loss of generality, and 
  if $ \vert b\vert^2>q $,
  $ C : c(z)=z\overline{z}+\overline{b}z+b\overline{z}+q=0 $ 
  represents a circle of radius  $ \sqrt{\vert b\vert^2-q} $ with center $ b $.

  Substituting $ a=0 $ and $ p=1 $ for $ F_4 $,
  we have
  \begin{align*}
  F_4(0,b,1,q,z)=& (2bz+b^2+1)\big((\overline{b}^2-1)z^4
          +(2\overline{b}q+(2\overline{b}^2-4)b)z^3+(6\overline{b}bq-6b^2)z^2\\
       &  +(2bq^2+2\overline{b}b^2q-4b^3)z
        +(b^2+1)q^2-2\overline{b}bq-b^4+\overline{b}^2b^2\big)=0.
  \end{align*}
  The first factor represents a single point.
  If the second factor is transformed by the similarity transformation 
  $ z=sw-b\ (s^2=\vert b\vert^2-q) $, the equation that gives the PA-point for 
  the unit circle is obtained as follows,
  \begin{equation}\label{eq:trans}
   (\vert b\vert^2-q)^2\big((\overline{b}^2-1)w^4
           -2\overline{b}\sqrt{\vert b\vert^2-q}w^3
          +2b\sqrt{\vert b\vert^2-q}w-b^2+1\big)=0.
  \end{equation}

  Then, by the transformation $ z=sw-b $, the foci $1$ and $-1$ 
  correspond to $(1+b)/s$ and $(-1+b)/s$, respectively.
  For these two points, the equation \eqref{eq:reflect} of the PA-point
  is
  \[
     (\overline{b}^2-1)z^4-2\overline{b}\sqrt{\vert b\vert^2-q}z^3
        +2b\sqrt{\vert b\vert^2-q}z-b^2+1=0.
  \]
  The above equation coincides with \eqref{eq:trans}.
  
  Thus, the equation $ F_4=0 $ is also valid for the case that 
  $ C $ is a circle.
\end{remark}

\begin{theorem}\label{thm:6ji}
  If the segment $ [-1,1] $ does not intersect with $ C $,
  the point $ z \in C $ such that the sum $ \vert z-1\vert +\vert z+1\vert  $ 
  is minimal is given as a root of the equation 
  $ F_4=0 $ of degree $ 6 $.
\end{theorem}

The following examples show how to find the reflection points.

\begin{example}
   Let $ D  =\{z\in\mathbb{C}\,:\, \vert z-2\vert 
   +\vert z-(1+2i)\vert >\sqrt{6}\} $.
   The PA-point $ u \in\partial D $  can be found by the following procedure. 

   The boundary $ C =\partial D $ is the ellipse written as 
   \[
        c(z)=(-3+4i)z^2-14z\overline{z}+(-3-4i)\overline{z}^2
          +(38-20i)z+(38+20i)\overline{z}-71=0. 
   \]
   In this case, the equation $ F_4=0 $ is given by
           \begin{align*}
             & (924-1232i)z^6+(-15308+7432i)z^5+(81677+2608i)z^4
              +(-189086-106196i)z^3\\
             & \quad +(185621+278356i)z^2+(-37976-281192i)z
              +(-29632+97824i)=0,
           \end{align*}
   and its roots are
       \[
        \begin{array}{lll}
        u_1\approx 1.923740-0.117041i, &  u_2\approx 1.772166+0.309916i, &
        u_3\approx 1.259144+0.426617i, \\
        u_4\approx 2.808489+0.435057i, &  u_5\approx 0.825548+1.934592i, &  
        u_6\approx 1.235845+2.067480i.
        \end{array}
      \] 
   It is easy to see that four roots $ u_1,u_3,u_5$, and $ u_6 $ are in $ C $, 
   and the function $ \vert z-1\vert +\vert z+1\vert  $ attains 
   its minimum at the point $ u_6\in C $ (see Figure \ref{pic:1} for details).
   Note that the case $ D=\{\vert z-2\vert +\vert z-(1+2i)\vert <\sqrt{6}\} $ 
   can be calculated exactly in the same way.

\begin{figure}
 \centerline{\fbox{\includegraphics[width=0.6\linewidth]{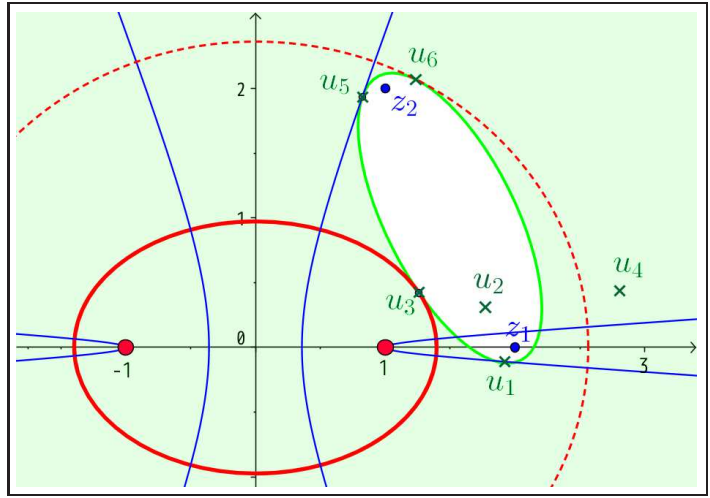}}}
  \caption{The leaning ellipse indicates $C$.
           The points $ u_3 $ and $ u_6 $ are tangent points of
           the thick and dotted ellipses with $ C $, respectively.
           The foci of these thick and dotted ellipses are both $ -1,1 $.
           Here, the points $ u_1 $ and $ u_5 $ are tangent points of
           $ C $ and the hyperbolas with foci $ -1,1 $.
           }\label{pic:1}
\end{figure}
\end{example}

\begin{example}
   Let $ D $ be the region given by 
   $ \{z\in\mathbb{C}\,:\, \big\vert \vert z-3\vert 
    -\vert z-(1+2i)\vert \big\vert >\sqrt{5}\} $.
   The PA-point $ u \in\partial D $  can be found by the following procedure. 

   The boundary $ C=\partial D $ is the hyperbolic curve defined by
   \[
      c(z)=8iz^2-4z\overline{z}-8i\overline{z}^2
          +(24-36i)z+(24+36i)\overline{z}-99=0. 
   \]
   In this case, the equation $ F_4=0 $ is given by 
           \begin{align*}
             &  6048z^6+(-66960-34992i)z^5+(212760+346428i)z^4
                 +(47268-1215900i)z^3\\
             & \quad +(-1363032+1675647i)z^2+(2156652-408726i)z
                  +(-850176-550557i)=0,
           \end{align*}
   and its roots are
     \[
       \begin{array}{lll}
         u_1\approx 2.542018-0.357669i,  & u_2\approx 2.645886+0.629896i, &
         u_3\approx 3.393387+0.463604i, \\
         u_4\approx 1.323205+1.940610i, &  u_5\approx 1.5i ,           &
         u_6\approx  1.166931+1.609271i.
      \end{array}
    \]
   It is easy to check that four roots  $ u_1,u_3 ,u_4, $ and $u_5$ are in 
   $ C $, and the function $ \vert z-1\vert +\vert z+1\vert  $ attains 
   its minimum at the point $ u_5\in C $ (see Figure \ref{pic:3} for details).
   Note that the case $ D=\{\big\vert \vert z-3\vert -\vert z
    -(1+2i)\vert \big\vert <\sqrt{5}\} $
   can be calculated just the same way.
\begin{figure}
 \centerline{\fbox{\includegraphics[width=0.6\linewidth]{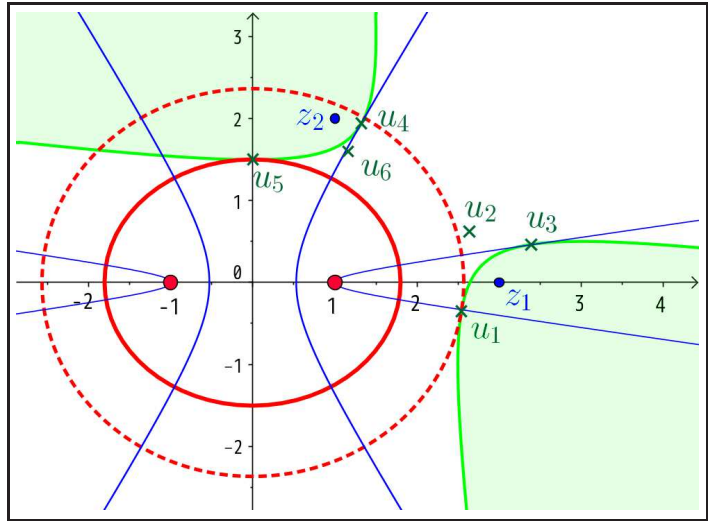}}}
  \caption{The leaning hyperbola indicates $C$.
           The two points $ u_1 $ and $ u_5 $ are tangent points of
           the dotted and thick ellipses with $ C $, respectively.
           The foci of these dotted and thick ellipses are both $ -1, 1 $.
           Moreover, the two points $ u_3 $ and $ u_4 $ are tangent points of
           $ C $ and the hyperbolas with foci $ -1,1 $.
          }\label{pic:3}
\end{figure}
\end{example}

\section{Triangular ratio metric on conic domains}\label{sect:smetric}

\begin{mysubsection}{\bf The procedure for calculating the triangular ratio distance}

For two points $ z_1,z_2\in\mathbb{C} $ and the domain $ D $ whose boundary 
is given by a conic $ \Gamma $, the triangular ratio metric
\[ 
   s_D(z_1,z_2)=\sup_{z\in \Gamma}\frac{\vert z_1-z_2\vert }
                            {\vert z_1-z\vert +\vert z-z_2\vert }
\]
is obtained as follows.
\end{mysubsection}
\begin{enumerate}
  \item Let $ C:\ c(z)=\overline{a}z^2+pz\overline{z}+a\overline{z}^2
         +\overline{b}z+b\overline{z}+q=0 $ 
         be the conic given by $ A(\Gamma) $, where $ A $ is the
        similarity transformation
        $ A(z)=2/(z_1-z_2)z-(z_1+z_2)/(z_1-z_2) $.
  \item For $ C $, solve the equation $ F_4(z)=0 $.
  \item Find the points $ \zeta\in C $ for which the minimum
        $ \min_{F_4(z)=0}\big\{\vert z+1\vert +\vert z-1\vert \big\} $ 
        is attained.
  \item Then, we have
        $ s_C(1,-1)=2/(\vert 1-\zeta\vert +\vert \zeta+1\vert ) $.
  \item Because a similarity transformation preserves 
        the quotient of distances, we have
        \[
           s_D(z_1,z_2)=\sup_{z\in \Gamma}\frac{\vert z_1-z_2\vert }
                {\vert z_1-z\vert +\vert z-z_2\vert }
                     =s_C(1,-1).
        \]
\end{enumerate}

\begin{mysubsection}{\bf Examples of level sets}

For a given domain $ G\subset\mathbb{C} $ and $ 0<t<1 $,
the set $ B_s(z,t)=\{\zeta\in G\,:\, s_G(z,\zeta)< t\} $,
is called the contour domain of the level $ t $.

\bigskip

Here, we will draw the level sets by solving the equation $F_4=0$. 
The used algorithm is the same as \cite[p.145 Algorithm]{fhmv}.
\end{mysubsection}

\begin{example} \

The left figure of Figure \ref{fig:contour1} indicates
the level sets $ B_s(0,t)=\{\zeta\in G\,:\, s_G(0,\zeta)<t\} $
for $ t=0.05, 0.1, \cdots,0.95, \mbox{and}\ 1 $ and the hyperbolic domain
$ G=\{\big\vert \vert z-(-1/2-1/2i)\vert 
-\vert z-(1-i)\vert \big\vert <4/5 \} $.
It seems that the edges of each contour curve are located on the set
\[ 
  \Big\{\big\vert \vert z-(-1/2-1/2i)\vert -\vert z-(1-i)\vert \big\vert 
  =\sqrt{1/2}\Big\}. 
\]

The right figure indicates
the level sets for the elliptic domain
$ G=\{\vert z-3/2\vert +\vert z-(-1/3-1/2i)\vert <11/5 \} $.
It seems that the edges of each contour curve are located on the set 
\[
   \Big\{\vert z-3/2\vert +\vert z-(-1/3-1/2i)\vert =(9+\sqrt{13})/{6}\Big\}.
\]

\begin{figure}
\centerline{\includegraphics[width=0.4\linewidth]{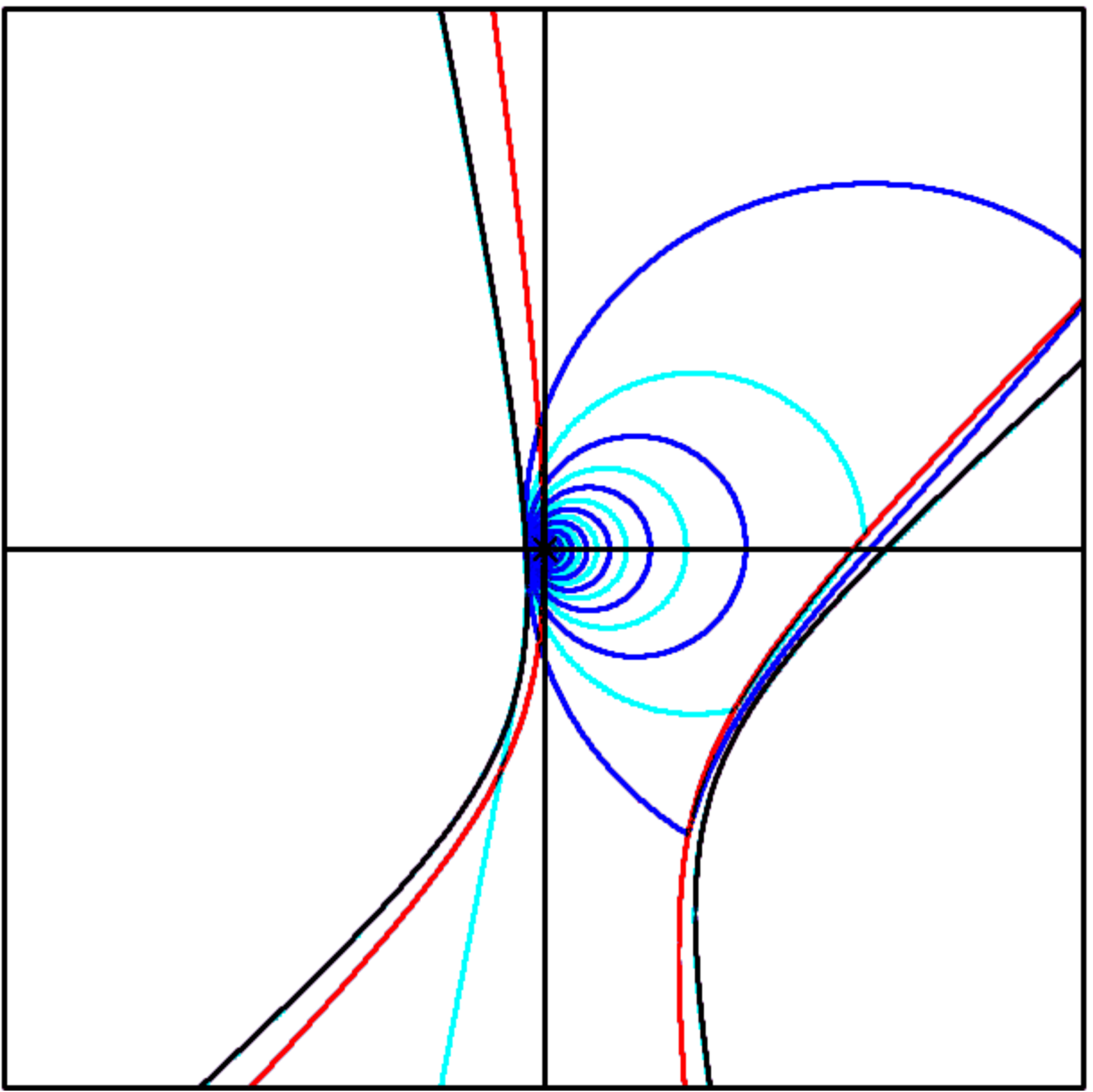}\qquad
            \includegraphics[width=0.4\linewidth]{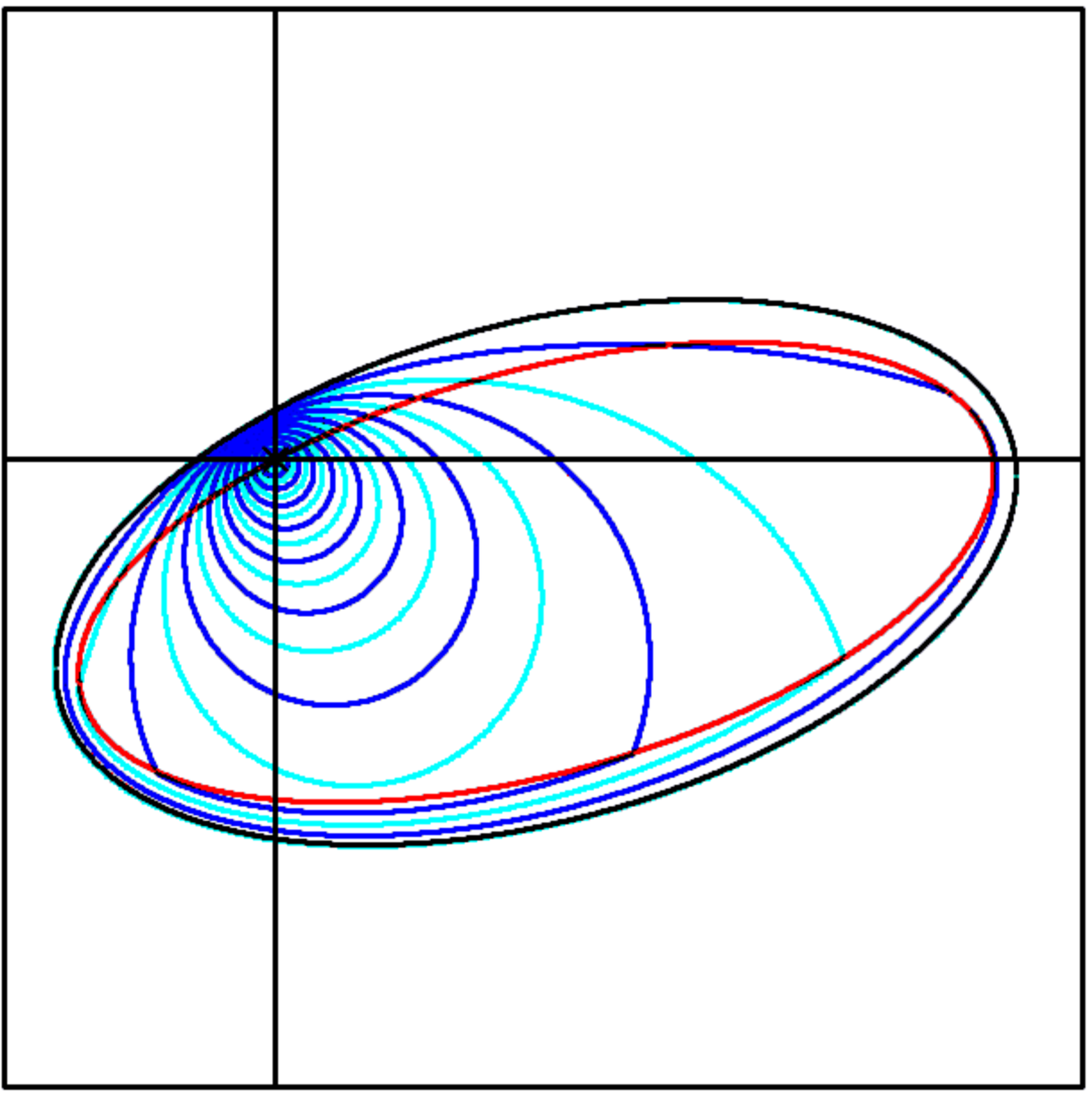}}
\caption{The level sets of 
         $ G=\{\big\vert \vert z-(-1/2-1/2i)\vert -\vert z
           -(1-i)\vert \big\vert <4/5 \} $ (left)
         and $ G=\{\vert z-3/2\vert +\vert z-(-1/3-1/2i)\vert <11/5 \} $
         (right).
         Note that each red curve passes through  
         all edge points of the contour curves.
        }
\label{fig:contour1}
\end{figure}
\end{example}

The above example leads to the following conjecture.

\begin{conjecture}
Let $ G $ be the domain defined by 
$ \{\vert z-f_1\vert +\vert z-f_2\vert <r\} $ or
$ \{\big\vert \vert z-f_1\vert -\vert z-f_2\vert \big\vert <r\} $, and
$ B_s(z,t)=\{\zeta\in G\,:\, s_G(z,\zeta)<t \}$.
Using a similarity transformation, we can set the center point $z $ to 0.
Then, the edge points of $ \partial B_s(0,t) $ are on the conic
$ \vert z-f_1\vert +\vert z-f_2\vert =\vert f_1\vert +\vert f_2\vert  $ or 
$ \big\vert \vert z-f_1\vert -\vert z-f_2\vert \big\vert 
  =\big\vert \vert f_1\vert -\vert f_2\vert \big\vert  $
if $ G = \{\vert z-f_1\vert +\vert z-f_2\vert <r\} $ or 
$ \{\big\vert \vert z-f_1\vert -\vert z-f_2\vert \big\vert <r\} $,
respectively.
\end{conjecture}

\section*{Acknowledgments}
This work was partially supported by JSPS KAKENHI Grant Number 19K03531 
and the Research Institute for Mathematical Sciences, 
an International Joint Usage/Research Center located in Kyoto University.


\end{document}